\documentclass[11pt]{amsart} 
\usepackage{amsmath,amsthm}
\usepackage{xypic}
\usepackage{enumerate}
\usepackage{lscape}
\usepackage{tikz-cd}
\usepackage{tikz}
\usepackage{color} 
\definecolor{darkred}{rgb}{1,0,0} 
\definecolor{darkgreen}{rgb}{0,1,0}
\definecolor{darkblue}{rgb}{0,0,1}
\usepackage{hyperref}
\hypersetup{colorlinks,
linkcolor=darkblue,
filecolor=darkblue,
urlcolor=darkred,
citecolor=darkgreen}

\usepackage{amscd}
\usepackage{amsfonts,amssymb,latexsym}
\usepackage[all]{xy}
\xyoption{arc}
\CompileMatrices

\newcommand{\udots}{\mathinner{\mskip1mu\raise1pt\vbox{\kern7pt\hbox{.}}
\mskip2mu\raise4pt\hbox{.}\mskip2mu\raise7pt\hbox{.}\mskip1mu}}

\newcommand{\SC}{{\mathcal{C}}}
\newcommand{\SE}{{\mathcal{E}}}
\newcommand{\SF}{{\mathcal{F}}}
\newcommand{\SK}{{\mathcal{K}}}
\newcommand{\SL}{{\mathcal{L}}}
\newcommand{\SM}{{\mathcal{M}}}
\newcommand{\SO}{{\mathcal{O}}}
\newcommand{\SP}{{\mathcal{P}}}
\newcommand{\ST}{{\mathcal{T}}}
\newcommand{\SX}{{\mathcal{X}}}
\newcommand{\SY}{{\mathcal{Y}}}

\newcommand{\PP}{\mathbb{P}}
\newcommand{\ZZ}{\mathbb{Z}}
\newcommand{\NN}{\mathbb{N}}
\newcommand{\CC}{\mathbb{C}}

\newcommand{\Ext}{\operatorname{Ext}}

\newcommand{\Spec}{\operatorname{Spec}}

\newcommand{\Hom}{\operatorname{Hom}}
\newcommand{\Sch}{\operatorname{Sch}}
\newcommand{\id}{\operatorname{Id}}

\newcommand{\too}{\longrightarrow}
\newcommand{\rk}{\operatorname{rk}}

\newcommand{\GL}{\operatorname{GL}}
\newcommand{\PGL}{\operatorname{PGL}}
\newcommand{\PSL}{\operatorname{PSL}}

\newcommand{\Pic}{\operatorname{Pic}}

\newcommand{\op}{\operatorname}

\newcommand{\svb}{\op{s}}
\newcommand{\ssvb}{\op{ss}}

\newcommand{\Div}{\op{Div}}
\newcommand{\Quot}{\op{Quot}}

\newtheorem{proposition}{Proposition}[section]
\newtheorem{theorem}[proposition]{Theorem}
\newtheorem{lemma}[proposition]{Lemma}

\newtheorem{corollary}[proposition]{Corollary}

\theoremstyle{definition}
\newtheorem{definition}[proposition]{Definition}

\numberwithin{equation}{section}

\begin{document}

\title[Isomorphisms between moduli stacks]{Isomorphisms between moduli stacks of vector bundles with fixed 
determinant}

\author[D. Alfaya]{David Alfaya}

\address{Department of Applied Mathematics and Institute for Research in Technology, ICAI School of Engineering,
Comillas Pontifical University, C/Alberto Aguilera 25, 28015 Madrid, Spain}

\email{dalfaya@comillas.edu}

\author[I. Biswas]{Indranil Biswas}

\address{Department of Mathematics, Shiv Nadar University, NH91, Tehsil Dadri,
Greater Noida, Uttar Pradesh 201314, India}

\email{indranil.biswas@snu.edu.in, indranil29@gmail.com}

\author[T. L. G\'omez]{Tom\'as L. G\'omez}

\address{Instituto de Ciencias Matem\'aticas (CSIC-UAM-UC3M-UCM),
Nicol\'as Cabrera 15, Campus Cantoblanco UAM, 28049 Madrid, Spain}

\email{tomas.gomez@icmat.es}



\date{}

\begin{abstract}
We classify all isomorphisms between moduli stacks of vector bundles of fixed determinant on a smooth
complex projective of genus at least $4$. It is shown
that each isomorphism between two 
different moduli stacks can be described as a composition of a pullback using an isomorphism of curves,
dualization of vector bundles and tensoring with the pullback of a line bundle on the curve. We finally compare 
the 2-group of automorphisms of the moduli stack of vector bundles
with the group of automorphisms of the moduli space of semistable vector bundles. 
\end{abstract}

\maketitle

\section{Introduction}

Let $X$ be a smooth irreducible complex projective curve. Given a fixed line bundle $\xi$ over $X$, let $\SM(X,r,\xi)$ denote the moduli stack of 
vector bundles $E$ of rank $r$ on $X$ together with an isomorphism $\det(E)\,:=\,\wedge^r(E)\,\stackrel{\cong}{\longrightarrow}
\, \xi$, and let $M^{\ssvb}(X,r,\xi)$ denote the
moduli scheme of semistable vector bundles with these data. Kouvidakis and Pantev proved that the automorphism group of $M^{\ssvb}(X,r,\xi)$ is 
generated by appropriate combinations of the following three ``basic transformations'' of vector bundles (see \cite{KP95}):
\begin{itemize}
\item Tensor $E$ by some line bundle $L$ of order $r$ on $X$: $E\,\longmapsto\, E\otimes L$.
\item Dualize the vector bundle and tensor with a suitable line bundle $L$ so that the
determinant is unchanged: $E\, \longmapsto\, E^\vee\otimes L$.
\item Take the pullback of $E$ using an automorphism $\sigma\,:\,X\,\longrightarrow\, X$ of the curve
and tensor $E$ or its dual with a suitable line bundle $L$ so that the determinant is
unchanged: $E\, \longmapsto \,L\otimes \sigma^*E$ or $E\, \longmapsto\, L\otimes\sigma^* E^\vee$.
\end{itemize}
This result was also proved with different techniques in \cite{HR04} and \cite{BGM13}, and in \cite{AB} a variation of
the argument from \cite{BGM13} was used to prove that all 2--birational maps between any pair of moduli schemes of
semistable vector bundles are also given by combinations of these three basic transformations (it was also shown there that
every 2--birational map between these moduli spaces actually extends to an isomorphism).

In this paper, we describe all possible isomorphisms between moduli stacks of vector bundles with fixed determinant. The isomorphisms between two moduli 
stacks is not a set, but a category, whose objects are morphisms
of groupoids (they are compatible with 
pull-backs) and whose morphisms are natural transformations between
morphisms of groupoids (Definition \ref{def:mapbetweenstacks}). In other words, the automorphism group of a moduli stack is
a 2-group.

We show that all possible isomorphisms between moduli stacks of vector bundles are isomorphic to the following maps which generalize the ``basic transformations'' to morphisms between stacks.
\begin{definition}
\label{def:basicTransformation}
Given curves $X$ and $X'$ with line bundles $\xi$ and $\xi'$ over $X$ and $X'$ respectively, and given
\begin{itemize}
\item an isomorphism $\sigma\,:X'\,\too\, X$,
\item a sign $s\,\in \, \{+1,\, -1\}$, and
\item a line bundle $L$ over $X$ together with an isomorphism $\xi'\,\cong\, \sigma^*(\xi\otimes L^r)^s$
\end{itemize}
we define the \emph{basic transformation} $\ST_{\sigma,L,s}:\SM(X,r,\xi) \longrightarrow \SM(X',r,\xi')$ as the following morphism of stacks. For every $T$-family of vector bundles $E\,\longrightarrow\, X\times T$:
$$\ST_{\sigma,L,s}(E)\ =\ \left \{ \begin{array}{ll}
(\sigma\times \id_T)^*( E\otimes \pi_X^*L ), & s\,=\,1,\\
(\sigma\times \id_T)^*(E^\vee \otimes \pi_X^*L ), & s\,=\,-1.
\end{array}\right.$$
On morphisms, $\ST_{\sigma,L,s}$ is defined as follows. Let $E'\,\longrightarrow\, X\times T'$ be 
another object, $f:T\to T'$ a morphism, 
and let $\varphi:E\stackrel{\cong}{\longrightarrow} (\id\times f)^*E'$ 
a morphism between the objects $E$ and $E'$ in the category $\SM(X,r,\xi)$. 
$$\ST_{\sigma,L,s}(\varphi)\ =\ \left \{ \begin{array}{ll}
(\sigma\times \id_T)^*( \varphi\otimes \id_{\pi_X^*L} ), & s\,=\,1,\\
(\sigma\times \id_T)^*((\varphi^{-1}){}^\vee \otimes \id_{\pi_X^*L} ), 
& s\,=\,-1.
\end{array}\right.$$
\end{definition}

Note that $\ST_{\sigma,L,s}$ is a covariant
functor. This is why we introduce the inverse $\varphi^{-1}$ in the last line. It is clear that $\ST_{\sigma,s,L}$
gives an isomorphism between the moduli stacks. The main result of this work is to to prove that these isomorphisms are the only ones up to 
isomorphism.

\begin{theorem}[{See Theorem \ref{thm:main}}]
\label{thm:intro}
Let $X$ and $X'$ be smooth complex projective curves of genus $g$ and $g'$ respectively, with $g,\, g'\, \ge\, 4$. Let $\xi$ and $\xi'$ be line bundles on $X$ and $X'$
respectively. If $\Psi\,:\, \SM(X,r,\xi) \,\longrightarrow \,\SM(X',r',\xi')$ is an isomorphism of stacks,
then $r\,=\,r'$, and there exist a basic transformation $\ST_{\sigma,L,s}$, for which $s$ can always be taken as $+1$ if $r\,=\,2$, and such that $\Psi$ is isomorphic to $\ST_{\sigma,L,s}$.
\end{theorem}

Let us consider for a moment the moduli stack of vector bundles
$\SM(X,r)$ where we don't fix the determinant (the connected
components of $\SM(X,r)$ are $\SM(X,r,d)$, indexed by the degree
$d\,\in\, \ZZ$ of vector bundles). A line bundle $\SL$ on $\SM(X,r)$ produces an
automorphism of $\SM(X,r)$ as follows: for each morphism $f\,:\,T\,\too\,
\SM(X,r)$, corresponding to a vector bundle $\SE_f$ on $X\times T$, 
we associate the morphism $f'\,:\, T\,\too \SM(X,r)$ corresponding to the
vector bundle $\SE'\,=\,\SE\otimes \pi_T^* f^*\SL$, where $\pi_T\, :\, X\times T\,
\longrightarrow\, T$ is the natural projection. It is easy to check
that this gives an automorphism of $\SM(X,r)$. Now let us return to the moduli
stack $\SM(X,r,\xi)$, parametrizing vector bundles $\SE$ on $X\times T$
together with an isomorphism $\det\SE\,\cong\, \pi_X^*\xi$. In the previous
construction we have $\det \SE'\,=\,\det \SE\otimes \pi_T^*f^*\SL^r$.
Therefore, the line bundle $\SL$ induces an automorphism on
$\SM(X,r,\xi)$ only if $\SL^r$ is trivial. However, the Picard group
of $\SM(X,r,\xi)$ is torsion-free, so $\SL\,\cong\,\SO_X$ if $\SL^r$ is trivial, and
consequently such an automorphism is trivial.

The idea of the proof of Theorem \ref{thm:intro} is to first show that an automorphism of
$\SM(X,r,\xi)$, which is not given by a combination of the three
``basic transformations'', has to be given by the previous
construction, and hence it has to be isomorphic to the trivial automorphism.

For this, following the ideas in \cite[\S~2]{ABGM}, we use ``beyond
GIT'' techniques (see \cite{Al,He,HL,AlHLHe}) to identify the locus
of semistable vector bundles inside the moduli stack $\SM(X,r,\xi)$.
In other words, it is shown that any automorphism of $\SM(X,r,\xi)$
preserves the open substack of semistable vector bundles and produces a
commutative diagram
$$
\xymatrix{
\SM^{\ssvb}(X,r,\xi) \ar[rr]^{\Psi^{\ssvb}} \ar[d] && \SM^{\ssvb}(X',r',\xi') \ar[d]\\
M^{\ssvb}(X,r,\xi) \ar[rr]^{\psi} && M^{\ssvb}(X',r',\xi').
}
$$
As mentioned before, it is known that the automorphism group of the moduli space is given
by a combination of the three ``basic transformations'' 
(cf. \cite{KP95,BGM13,AB}), so composing with such a combination, it
may be assumed that $\psi$ is actually the identity map. The identity map preserves
the open subscheme of stable bundles, and then we have a diagram
$$
\xymatrix{
\SM^{\svb}(X,r,\xi) \ar[rr]^{\Psi^{\svb}} \ar[d] && \SM^{\svb}(X,r,\xi) \ar[d]\\
M^{\svb}(X,r,\xi) \ar@{=}[rr] && M^{\svb}(X,r,\xi).
}
$$
We then prove (see Lemma \ref{lemma:tensorLineBundle}) that $\Psi^{\svb}$ has to be
trivial. Finally it is shown that an automorphism of $\SM^{\svb}(X,r,\xi)$
extends uniquely to $\SM(X,r,\xi)$.

Notice that, as the moduli stack of vector bundles is not separated,
the uniqueness of the extension does not follow automatically from the
fact that the open substack of stable bundles is dense. To prove the
desired extension property, we prove that stable vector bundles are
not only dense in the moduli stack, but that there exists an
additional ``2-dimensional density'' property as follows 
(see Corollary \ref{cor:2approximation} and see 
Lemma \ref{lemma:2approximation} for a version of 
this results for families of vector bundles).

\begin{theorem}
\label{thm:2approxIntro}
For any vector bundle $E$ on a smooth projective curve of genus $g\,\ge\, 2$, there exists a family of vector bundles $\SE$
on $X\times T$ for a certain connected smooth 2-dimensional scheme $T$ and a point $t_0\,\in\, T$, such
that $\SE\big\vert_{X\times \{t\}}$ is stable for each $t\,\ne\, t_0$ and $\SE\big\vert_{X\times \{t_0\}}\,\cong\, E$.
\end{theorem}

This new notion of ``2-dimensional density'' allows us to prove that maps between the moduli substacks of stable 
vector bundles extend uniquely to isomorphisms between the entire stacks, thus finishing the proof.

The paper is structured as follows. In Section \ref{section:groupoids} we 
recall some generalities on the moduli stack and the moduli 
scheme of vector bundles with fixed determinant
and show that the 
category of maps from $\SM$ to $\SM$ is equivalent to the category of vector bundles on $X\times \SM$. 
Section \ref{section:basicTransformations} describes the basic transformations of families of 
vector bundles. Section \ref{section:beyondGIT} shows the recovery of the moduli substack and moduli scheme 
of semistable vector bundles from the isomorphism class of the stack of vector bundles. Section 
\ref{section:extension} includes the proof of the ``2-dimensional density'' property and the proof of the 
main result, and in Section \ref{section:rational} we give an example to show that, in the case of a rational curve, the automorphism groups of the moduli space and stack are quite different.

\section{groupoids}\label{section:groupoids}

Let us start by describing some general theoretical results on moduli stacks and, concretely, on moduli stacks of vector bundles on a curve, which will be useful through the work. Through this work, let $X$ be a smooth irreducible complex projective curve of genus at least 2. Let $\xi$ be a line bundle on $X$ and let $\SM(X,r,\xi)$ denote the moduli stack of vector bundles on $X$ with fixed determinant $\xi$. Concretely, for any scheme $T$, the groupoid of morphisms $f\,:\,T\,\longrightarrow\, \SM\,=\,\SM(X,r,\xi)$
to the moduli stack is equivalent to the groupoid of pairs $(E,\,\beta)$, 
where $E$ is a vector bundle on $X\times T$ and $\beta$ is an isomorphism
between $p_X^*\xi$ and $\det(E)$. In this section we prove that this continues to
hold even if $T$ is replaced by an algebraic stack $\ST$. 
The proof is technical, nevertheless
the idea is easy: given an atlas $U\,\longrightarrow\, \ST$, a morphism
$f\,:\,\ST\,\longrightarrow\,\SM$ can be described by a morphism $f_U\,:\,U\,\longrightarrow\, \SM$ together 
with some descent data. Since $U$ is a scheme, the morphism $f_U$ corresponds to
a vector bundle $\SE_U$ on $U\times \SM$, and the descent data for
$f_U$ gives descent data for $\SE_U$, thus obtaining a vector bundle
$\SE$ on $T\times \SM$. Therefore, setting $\ST\,=\,\SM$, 
we obtain an equivalence between
the category of automorphisms of $\SM$ and the category of vector
bundles on $X\times \SM$.

A \emph{category over} $\SC$ is a functor of categories $p\,:\,\SX \,\longrightarrow\, \SC$.
Given an object $B$ in $\SC$, we define the fiber category $\SX(B)$ as the
subcategory of $\SX$ with objects $u$ such that $p(u)\,=\,B$ (this is actual equality,
not just isomorphism) and morphisms $\varphi$ such that $p(\varphi)\,=\,id_B$.

A morphism $\varphi\,:\,u_2\,\longrightarrow\, u_1$ in a category $\SX$ over $\SC$ 
is called Cartesian 
if for any diagram
$$
\xymatrix{
{u_3} \ar@{|->}[d] \ar@(ru,lu)[rr]^{\psi} \ar@{-->}[r]^{\gamma} &
{u_2} \ar@{|->}[d]\ar[r]^{\varphi} & {u_1} \ar@{|->}[d] \\
{B_3} \ar[r]^{g} & {B_2} \ar[r]^{f} & {B_1 \; ,} \\
}
$$
where $p(u_i)\,=\,B_i$, $p(\varphi)\,=\,f$, $p(\psi)\,=\,f\circ g$,
there exists a unique morphism $\gamma\,:\,u_3 \,\longrightarrow\, u_2$ with
$p(\gamma)\,=\,g$ and $\psi\,=\,\varphi\circ
\gamma$.

If $\varphi$ is Cartesian, then we 
say that $u_2$, together with the morphism $\varphi$, is 
a pullback of $u_1$ along $f$. If a pullback exists, then it is unique
up to a unique isomorphism, i.e., if $u'_2$ is another 
pullback, then there exists a unique isomorphism $\lambda\,:\,u_2\,\longrightarrow\, u'_2$
with $p(\lambda)\,=\,\id_{B_2}$ and $\varphi'\circ \lambda\,=\,\varphi$.

A \emph{fibered category} is a category $p\,:\,\SX\,\longrightarrow\, \SC$ over $\SC$, 
such that at least one pullback exists for each object $u_1$ in $\SX(B_1)$ and 
morphism $f\,:\,B_2\,\longrightarrow\, B_1$. A \emph{cleavage} (also called
\emph{splitting}) of a fibered category is
a choice of a unique pullback for every object $u_1$ and morphism $f$
(see \cite[\S~3.3]{Ol16}).

Denote by $\op{Cat}^{\op{cart}}_{/\SC}$ the 2-category of 
fibered categories over $\SC$, defined as follows:

\begin{definition}[{\cite[Section 0012]{HL25}, \cite[Definition 3.1.3]{Ol16}}]\label{def:mapbetweenstacks}
Let $\SX$ and $\SY$ be categories fibered over a category $\SC$.
The category $\op{Map}_{\op{Cat}^{\op{cart}}_{/\SC}}(\SX,\,\SY)$ of morphisms of fibered categories 
between $\SX$ and $\SY$ is defined as follows. 
\begin{itemize}
\item An object is a functor $f\,:\,\SX\,\longrightarrow\, \SY$ which preserves Cartesian morphisms.
\item A morphism between functors $f$ and $g$ is a natural transformation
$\eta\,:\,f\,\Rightarrow\, g$ such that, for every $\xi\,\in \,\SX$, the associated
morphism $\eta_\xi\,:\,f(\xi)\,\longrightarrow\, g(\xi)$ lies over the identity morphism.
\end{itemize}
\end{definition}

A \emph{category fibered in groupoids} is a fibered category such
that the fiber $\SX(B)$ is a groupoid for all objects $B$ in $\SC$,
i.e., all arrows mapping to the identity morphism are isomorphism.
We denote by $\op{Cat}^{\op{cart},\cong}_{/\SC} \ \subset\ 
\op{Cat}^{\op{cart}}_{/\SC}$ the full 2-subcategory of categories fibered
in groupoids.

Let $\SX$ be an algebraic stack over the fppf site $\SC\,=\,\Sch_{/S}$ with 
schematic diagonal
(i.e., the diagonal $\Delta\,:\,\SX\,\longrightarrow\, \SX\times_\SC\SX$ is representable
by a scheme). In this article we are interested in the case $S\,=\,\Spec \CC$, 
but the results of this section are valid over any scheme $S$.
Let $u\,:\,U_0\,\longrightarrow\, \SX$ be an atlas. We get a diagram of schemes (they
are schemes because the diagonal is representable by a scheme)
\begin{equation}\label{eq:simplicial}
\xymatrix{
{U_2:=U_0\times_\SX U_0 \times_\SX U_0} \ar@<2ex>[r]^-{p_0} \ar[r]^-{p_1} \ar@<-2ex>[r]^-{p_2} &
{U_1:=U_0\times_\SX U_0} \ar@<1ex>[r]^-{q_0} \ar@<-1ex>[r]^-{q_1} 
&
{U_0} 
},
\end{equation}
where the arrows 
are given by the natural projections.
Note that the points\footnote{A point in an $S$-scheme $V$ is 
a morphism of $S$-schemes $B\,\longrightarrow\, V$. } of $U_1$ are triples $(x_0,\,f_1,\,x_1)$, where
$x_i$ are points in $U_0$ and $f_1\,:\,u(x_1)\,\longrightarrow\, u(x_0)$ is a morphism in
the category $\SX$ between the corresponding objects in $\SX$. 
The images of $q_0$ and $q_1$ are, respectively, $x_1$ and $x_0$ 
(i.e., $q_i$ forgets the point $x_i$), so
these projections can be thought of as ``source'' and ``target''.
The points of $U_2$ are tuples $(x_0,\,f_1,\,x_1,\,f_2,\,x_2)$, where $x_i$ 
are points in $U_0$, and $f_i\,:\,u(x_i)\,\longrightarrow\, u(x_{i-1})$ are morphisms in 
the category $\SX$. Each of the three arrows from $U_2$ to $U_1$ 
can be thought of as forgetting one of the points, so the images of
$p_0$, $p_1$ and $p_2$ are,
respectively, $(x_1,\,f_2,\,x_2)$, $(x_0,\, f_1\circ f_2,\, x_2)$ 
and $(x_0,\,f_1,\,x_1)$. Note that the
morphism $p_1$ can be thought of as a ``composition''.

The morphisms satisfy the following simplicial identities:
\begin{eqnarray*}
q_ip_j&\,=\,& q_{j-1}p_i \quad \text{if $i\,<\, j$}\\
q_ip_j&\,=\,& q_jp_{i+1} \quad \text{if $i\,\geq \,j$.}
\end{eqnarray*}
Given a morphism of stacks $F\,:\,\SX\,\longrightarrow\, \SY$, the composition with the
atlas $u\,:\,U_0\,\longrightarrow\, \SX$ gives a morphism 
$$
\xi\ =\ F\circ u\ :\ U_0\ \too\ \SY,
$$ 
which we can interpret
as an object in $\SY(U_0)$. Pullbacks by $q_0$ and $q_1$ give two 
objects in $\SY(U_1)$. Since $\xi$ factors through $\SX$, there is a 
natural isomorphism between them
$\sigma\,:\,q_0^*\xi \,\longrightarrow\, q_1^*\xi$ which we think of as gluing data. 
The simplicial identities imply that the glueing data satisfy the
following cocycle condition:
$$
p_0^* \sigma \circ (p_1^*\sigma)^{-1} \circ p_2^* \sigma \ =\ \id .
$$
Summing up, a morphism of stacks from $\SX$ to $\SY$ 
produces a morphism $\xi$ from 
the atlas $U_0$ to $\SY$ together with gluing data $\sigma$ satisfying the
cocycle condition. The following proposition says that the converse holds.

\begin{proposition}[{\cite[Section 005J]{HL25}}]\label{prop:mapxy}
Let $\SY$ be a stack over $\SC\,=\,\Sch_{/S}$, and let $\SX$ be an algebraic
stack with a scheme atlas $U_0\to \SX$ and schematic diagonal.
The category $\op{Map}_{\op{Cat}^{\op{cart}}_{/\SC}}(\SX,\,\SY)$ 
is equivalent to the category $\op{Desc}(U_0\to \SX,\,\SY)$
of descent data, defined as follows.
\begin{itemize}
\item An object is a pair $(\xi,\,\sigma)$, where $\xi$ is an object in $\SY(U_0)$ and $\sigma
\,:\,q_0^*(\xi)\,\longrightarrow\, q_1^*(\xi)$ is an isomorphism in $\SY(U_1)$
which satisfies the cocycle condition 
$$
p_0^* \sigma \circ (p_1^*\sigma)^{-1} \circ p_2^*\sigma\ =\ \id .
$$

\item A morphism $(\xi,\,\sigma)\,\longrightarrow\, (\xi',\,\sigma')$ is a morphism 
$\eta\,:\,\xi\,\longrightarrow\, \xi'$ in $\SY(U_0)$ such that 
$q_1^*(\eta)\circ\sigma\,=\,\sigma'\circ q_0^*(\eta)$ in $\SY(U_1)$.
\end{itemize}
\end{proposition}

\begin{proof}
We give a sketch. For more details, 
see \cite[Section 004Y]{HL25}.

First consider $U_\bullet$ as a fibered category, whose objects and arrows over $T$ are
$U_0(T)$ and $U_1(T)$ respectively. Then the category of descent data is
the category of morphisms 
$$
\op{Desc}(U_0 \to \SX,\,\SY)
\ =\ \op{Map}_{\op{Cat}^{\op{cart}}_{/\SC}}(U_\bullet,\,\SY).
$$
The atlas morphism $U_0\,\longrightarrow\, \SX$ induces
a functor of categories $U_\bullet\,\longrightarrow\, \SX$.
The category $U_\bullet$ is not a stack (because it is
only a ``presheaf'' but not a ``sheaf'', i.e., it is a 
prestack but not a stack). Then a stackification 
$U_\bullet\,\longrightarrow\, BU_\bullet$ is constructed
(see \cite[Definition 0055]{HL25}); the
canonical morphism $BU_\bullet\,\longrightarrow\, \SX$ is an equivalence
(see \cite[Proposition 005B]{HL25}), and the composition of morphisms
$$U_\bullet\ \longrightarrow\ BU_\bullet\ \longrightarrow\ \SX$$ induces the
equivalence between the categories 
$\op{Map}_{\op{Cat}^{\op{cart}}_{/\SC}}(\SX,\,\SY)$ and $\op{Desc}(U\to\SX,\,\SY)$.
\end{proof}

Let $\SX$ be an algebraic stack over $\SC\,=\,\Sch_{/S}$ with
schematic diagonal and consider the diagram of schemes in
\eqref{eq:simplicial}. There are different ways of defining the
category of quasi-coherent sheaves over $\SX$. We will sketch the
definition in \cite[Section 005Q]{HL25}.

Define the fibered category $p\,:\, \op{QCoh_{/S}}\,\longrightarrow\, \Sch_{/S}$ as the
category whose objects are pairs $(X,\,E)$, where $X$ is an $S$-scheme
and $E$ is a quasi-coherent sheaf on $X$, and a morphism of it to $(Y,\,F)$
is a pair $(f,\,\psi)$, where $f\,:\,X\,\longrightarrow\, Y$ is a morphism of $S$-schemes 
and $\psi\,:\,f^*E \,\longrightarrow\, F$ is a homomorphism 
(see \cite[Definition 0016]{HL25}). Then define
$$
\op{QCoh}(\SX)\ :=\ \op{Map}_{\op{Cat}^{\op{cart}}_{/\SC}}(\SX,\,\op{QCoh_{/S}}).
$$

Given a quasi-coherent sheaf on $\SE$ on $\SX$ 
and an atlas $u\,:\,U_0\,\longrightarrow\, \SX$,
we obtain a sheaf $E\,=\,u^*\SE$ on $U_0$ and an isomorphism
$\sigma\,:\,q_0^*E\,\longrightarrow\, q_1^*E$ that satisfies the cocycle condition.
The following proposition says that the converse also holds.

\begin{proposition}[{\cite[Section 005R]{HL25}}]\label{prop:qcohx}
Let $\SX$ be an algebraic
stack with a scheme atlas $U_0\,\longrightarrow\, \SX$ and schematic diagonal.
The category $\op{QCoh}(\SX)$ 
is equivalent to the category $\op{Desc}(U_0\to \SX)$
of descent data, defined as follows:
\begin{itemize}
\item An object is a pair $(\xi,\,\sigma)$, where $\xi$ is a
quasi-coherent sheaf on $U_0$ and $\sigma\,:\,q_0^*(\xi)\,\longrightarrow\, q_1^*(\xi)$ 
is an isomorphism satisfying the cocycle condition 
$$
p_0^* \sigma \circ (p_1^*\sigma)^{-1} \circ p_2^*\sigma\ =\ \id .
$$

\item A morphism $(\xi,\,\sigma)\,\longrightarrow\, (\xi',\,\sigma')$ is a homomorphism 
$\eta\,:\,\xi\,\longrightarrow\, \xi'$ of quasi-coherent sheaves on $U_0$ such that 
$q_1^*(\eta)\circ\sigma\,=\,\sigma'\circ q_0^*(\eta)$ in $\SY(U_1)$.
\end{itemize}
\end{proposition}

\begin{proof}
The proof is similar to the proof of Proposition \ref{prop:mapxy}. We
first interpret the descent datum as a morphism from a category $U_\bullet$
to $\op{QCoh_{/S}}$. We then construct the stack $BU_\bullet$, which
is a category equivalent to $\SX$, and use \cite[Proposition 005B]{HL25}
to conclude that the category of descent datum is equivalent
to $\op{QCoh}(\SX)$. See \cite[Section 005R]{HL25} for more details.
\end{proof}

Let us define the category $\op{VBun}_{r,\xi}(X\times \SM)$.
The objects are pairs $(E,\,\beta)$ where 
$E$ is a vector bundle on $X\times \SM$ and $\beta:\det(E)\to p_X^*\xi$ is an isomorphism. A
morphism $\eta\,:\,(E,\,\beta) \,\Rightarrow\, (E',\,\beta')$ is an
isomorphism $f:E\cong E'$ such that $\beta'\circ\det(f)\,=\,\beta$.
We can also describe this category using descend data as follows.

\begin{corollary}\label{cor:vbunxm}
Let $U_0\to \SM(X,r,\xi)$ be an atlas, and consider the induced atlas on the
product $U_0\times X\to X\times \SM$. The category $\op{VBun}_{r,\xi}(X\times \SM)$ is equivalent to the category 
$\op{Desc}^{\rm VBun}_{r,\xi}(X\times U_0\to X\times \SM)$ defined as follows
\begin{itemize}
\item An object is a tuple $(\xi,\,\beta ,\,\sigma)$, where $\xi$ is a
rank $r$ vector bundle on $X\times U_0$, $\beta$ is an isomorphism
between $\det(\xi)$ and $p_X^*\xi$, and 
$\sigma\,:\,q_0^*(\xi)\,\longrightarrow\, q_1^*(\xi)$ 
is an isomorphism respecting the isomorphism $\beta$ and 
satisfying the cocycle condition 
$$
p_0^* \sigma \circ (p_1^*\sigma)^{-1} \circ p_2^*\sigma\ =\ \id .
$$

\item A morphism $(\xi,\,\beta,\,\sigma)\,\longrightarrow\, (\xi',\,\beta',\sigma')$ is an isomorphism 
$\eta\,:\,\xi\,\longrightarrow\, \xi'$ of vector bundles on $U_0$ (compatible with the isomorphisms $\beta$ and $\beta'$) such that 
$q_1^*(\eta)\circ\sigma\,=\,\sigma'\circ q_0^*(\eta)$ in $\SY(U_1)$.
\end{itemize}

\end{corollary}

\begin{proof}
This is a straightforward consequence of Proposition \ref{prop:qcohx}.
\end{proof}

We now have the following.

\begin{lemma}\label{lem:equivautovb}
Let $X$ be a complex smooth projective curve, and
let $\SM$ be the moduli stack of vector bundles on
$X$ with fixed rank $r$ and determinant $\xi$.

There is a natural equivalence between the category of
automorphisms of $\SM$ (as fibered categories over $\Sch_{/\CC}$) and
the category of vector bundles on $X\times \SM$ of rank $r$
and fixed determinant $p^*_{X}\xi$
$$
\op{Map}_{\op{Cat}^{\op{cart}}_{/\SC}}(\SM,\, \SM)\ 
\cong\ 
\op{VBun}_{r,\xi}(X\times \SM)\ .
$$
\end{lemma}

\begin{proof}
By Proposition \ref{prop:mapxy}, we have an equivalence of categories
$$
\op{Map}_{\op{Cat}^{\op{cart}}_{/\SC}}(\SM,\,\SM) \ \cong \ 
\op{Desc}(U_0\to \SM,\,\SM).
$$
By definition of the stack of vector bundles $\SM$, a morphism $f:U_0\to \SM$
is the same thing as a rank $r$ vector bundle on $X\times \SM$ with fixed
determinant $p_X^*\xi$, and a morphism between two maps $f$ and $f'$ is
an isomorphism between the corresponding vector bundles, respecting the 
fixed determinant. It follows that we have an equivalence of categories
between the following descent data
$$
\op{Desc}(U_0\to \SM,\,\SM) \ \cong \ 
\op{Desc}^{\rm VBun}_{r,\xi}(X\times U_0\to X\times \SM).
$$
Finally, Corollary \ref{cor:vbunxm} gives an equivalence of categories
$$
\op{Desc}^{\rm VBun}_{r,\xi}(X\times U_0\to X\times \SM) \ \cong\
\op{VBun}_{r,\xi}(X\times \SM).
$$
This completes the proof.
\end{proof}

Under the equivalence in Lemma \ref{lem:equivautovb}, the identity morphism corresponds to
the Poincar\'e bundle $(\SP,\,\beta\,:\,\det(\SP)\,\cong\, p^*_X\xi)$ 
on $X\times \SM$, and a morphism
$f\,:\,\SM\,\too\, \SM$ corresponds to the pair 
$(\SP_f\,=\,(\id_X\times f)^*\SP,\ \beta_f\,=\,(\id_X\times f)^*\beta)$.

Conversely, let $f_{\SE}$ be the automorphism of $\SM$ corresponding to a vector 
bundle $\SE$ on $X\times \SM$. Then the composition law is
$$
f_{(\id_X\times f_{\SE'})^*\SE}\ =\ f_{\SE}\circ f_{\SE'}.
$$

Recall that a vector bundle $E$ is called \emph{stable} (respectively \emph{semistable}) if for any proper subbundle $F\, \not=\, 0$ of $E$ the following inequality holds:
$$ \frac{\deg(F)}{\rk(F)} \,<\, \frac{\deg(E)}{\rk(E)} \quad \quad (\text{respectively, }\,\le\, ).$$
Denote by $\SM^{\svb}(X,r,\xi)$ and $\SM^{\ssvb}(X,r,\xi)$ the substacks of $\SM(X,r,\xi)$ parametrizing the stable and semistable vector bundles respectively. As a consequence of \cite[p.~635, Theorem 2.8(B)]{Ma}, both of them are open substacks.

Semistable vector bundles of rank $r$ and determinant $\xi$ admit a moduli scheme, which will be denoted by $M^{\ssvb}(X,r,\xi)$. Analogously, by \cite[p.~635, Theorem 2.8(B)]{Ma}, the locus of stable vector bundles is parametrized by an open subscheme of $M^{\ssvb}(X,r,\xi)$ that will be denoted by $M^{\svb}(X,r,\xi)$.

\begin{proposition}\label{prop:autmap}
If the genus of $X$ is at least 3, then the automorphism group of the Poincar\'e bundle $(\SP,\,
\beta)$ with fixed determinant is the cyclic group $\ZZ/r\ZZ$. The same holds for any vector
bundle $(\SP_f,\,\beta_f)$ obtained from the Poincar\'e bundle by an automorphism $f$ of $\SM$.
\end{proposition}

\begin{proof}
For simplicity, denote the substack of stable vector bundles by $\SM^{s}\subset \SM$ . It is
an open substack inside the smooth stack $\SM$, whose complement has 
codimension at least 2 (this follows from Lemma \ref{lemma:codimDiv}), 
so we can apply Hartogs' theorem
$$
H^0(X\times \SM,\, ad\,\SP)\ =\ H^0(X\times \SM^{s},\, ad\,\SP^s)
$$
(where $\SP^s$ is the restriction to the stable part $X\times \SM^{s}$).
The moduli scheme $M^{s}$ is a good moduli (in the sense of \cite{Al}), 
hence $\pi_*\SO_{\SM^{s}}\,=\,\SO_{M^{s}}$ where $\pi\,:\,\SM^{s}\,\too\, M^{s}$ 
is the morphism from the moduli stack of 
stable vector bundles to the moduli space. Furthermore, the morphism $\pi$
is a gerbe, the action of scalars on $ad\SP$ is trivial, so 
$ad\,\SP^s$ descends to $X\times M^{s}$. This means that there is a 
vector bundle $A$ on $X\times M^{s}$ and an isomorphism 
$(\id_X\times \pi)^*A\,\cong\, ad\,\SP^s$. Consequently,
$$
H^0(X\times \SM^{s},\, ad\,\SP^s)\,=\, 
H^0(X\times M^{s},\, \pi_* ad\,\SP^s)\,=\, 
H^0(X\times M^{s},\, A\otimes \pi_*\SO_{X\times \SM^{s}})\,=\, 
H^0(X\times M^{s},\, A). 
$$
Given that $g\,\geq\, 3$, we can use \cite[Theorem 3.10 (iii)]{BBN} to 
conclude that $A$ is slope stable
with respect to any polarization in $X\times M^{s}$. 
Since the 
degree of $A$ is zero, the slope stability of $A$ implies that
it has no global sections.

Summing up, we have $H^0(X\times \SM,\, ad\,\SP)\,=\,0$. Then 
$H^0(X\times \SM,\,End\,\SP)\,=\,\CC$, and the
automorphism group of $\SP$ is the multiplicative group $\CC\setminus\{0\}$.
Therefore, the automorphism group of $(\SP,\,\beta)$ is 
$\ZZ/r\ZZ$, because the determinant has to be fixed.

It is clear that the same holds for any vector bundle with fixed determinant
$(\SP_f,\,\beta_f)$ obtained using an automorphism $f$ of $\SM$.
\end{proof}

\section{Actions of basic transformations on the moduli spaces}\label{section:basicTransformations}

Recall that given a vector bundle $E$ on $X$, there are
the following three natural ways to produce another vector bundle on $X$:
\begin{itemize}
\item Tensor $E$ by some line bundle $L$ on $X$: In other words, $E\, \longmapsto\, E\otimes L$.
\item Dualize the bundle:\, $E\, \longmapsto\, E^\vee$.
\item Take the pullback of $E$ using an automorphism $\sigma\,:\,X\,\longrightarrow\, X$ of the curve:
$E\, \longmapsto\, \sigma^*E$.
\end{itemize}
More generally, following the notation in \cite{AB}, given a line bundle $L$ over $X$, an isomorphism of curves $\sigma\,:\,X'
\,\stackrel{\sim}{\longrightarrow}\, X$ and a sign $s\,\in\, \{1,\, -1\}$, we can define the \emph{basic transformation} $\ST_{\sigma,L,s}$ of a vector bundle $E$
of rank $r$ as
\begin{equation}\label{a1}
\ST_{\sigma,L,+}(E)\ :=\ \sigma^*(E\otimes L) \quad \quad \text{for }s\,=\,1,
\end{equation}
\begin{equation}\label{a2}
\ST_{\sigma,L,-}(E)\ :=\ \sigma^*(E\otimes L)^\vee \quad \quad \text{for }s\,=\,-1.
\end{equation}

Choose a line bundle $\xi$ on $X$. Let $\SM(X,r,\xi)$ be the moduli stack of vector bundles on $X$
of rank $r$ and determinant $\xi$. The above transformations \eqref{a1} and \eqref{a2} are well defined over families of vector bundles with fixed determinant $\xi$ and, for each choice of the triple $(\sigma, \, L, \,s)$, the transformation $\ST_{\sigma,L,s}$ induces the isomorphism between moduli stacks
$$\ST_{\sigma,L,s}\ :\ \SM(X,r,\xi)\ \stackrel{\sim}{\longrightarrow}\ \SM(X',r,\sigma^*(\xi\otimes L^r)^s)$$
described in Definition \ref{def:basicTransformation} at the introduction.

The operations of pullback, tensor product and taking dual preserve both the stability and semistability of a vector bundle.
Consequently, each of the two maps in \eqref{a1} and \eqref{a2} preserves
both the substacks of semistable and stable vector bundles and they induce isomorphisms
$$\ST_{\sigma,L,s}^{\op{sch}}\ :\ M^{\ssvb}(X,r,\xi) \ \stackrel{\sim}{\longrightarrow}\ M^{\ssvb}(X',r,\sigma^*(\xi\otimes L^r)^s)$$
between the corresponding moduli schemes of semistable vector bundles, which, moreover, preserve the locus of stable vector bundles.

By \cite{KP95} or \cite{BGM13}, the automorphism group of the moduli scheme $M^{\ssvb}(X,r,\xi)$ is given by 
the group of basic transformations $\ST_{\sigma,L,s}^{\op{sch}}$ such that $\sigma\,\in\, \op{Aut}(X)$ and 
$\sigma^*(\xi\otimes L^r)^s \,\cong\, \xi$. Moreover, by \cite[Theorem 2.11]{AB}, each isomorphism of schemes 
$\Psi^{\op{sch}}\,:\,M^{\ssvb}(X,r,\xi) \,\longrightarrow\, M^{\ssvb}(X',r',\xi')$ between two different 
moduli schemes of semistable vector bundles must necessarily be a basic transformation $\ST_{\sigma,L,s}$ for 
some isomorphism $\sigma\,:\,X'\,\longrightarrow \,X$ between the corresponding curves. More precisely, the 
following holds by \cite{KP95,BGM13,AB}.

\begin{lemma}[{\cite[Theorem B]{KP95}, \cite[Theorem 1.1]{BGM13}, \cite[Theorem 2.11]{AB} and \cite[Lemma 2.1]{AB}}]
\label{lemma:autoScheme}
Let $X$ and $X'$ be smooth projective curves of genus $g$ and $g'$ respectively such that $g,\,g'\,\ge\, 4$. Let $\xi$ and $\xi'$ be line bundles over $X$ and $X'$
respectively. If $\Psi^{\op{sch}}\,:\,M^{\ssvb}(X,r,\xi) \,\longrightarrow\, M^{\ssvb}(X',r',\xi')$ is an
isomorphism of schemes, then $r\,=\,r'$, and there exist
\begin{itemize}
\item an isomorphism $\sigma\,:\,X'\,\stackrel{\sim}{\longrightarrow}\, X$,
\item a sign $s\,\in\, \{1,\,-1\}$, and
\item a line bundle $L$ on $X$ such that $\xi' \,\cong\, \sigma^*(\xi\otimes L^r)^s$,
\end{itemize}
satisfying the condition that for each S-equivalence class $[E]$ of a semistable vector bundle in the moduli scheme $M^{\ssvb}(X,r,\xi)$,
$$\Psi^{\op{sch}}([E])\ \cong\ [\ST_{\sigma,L,s}(E)].$$
Moreover, if $r\,=\,2$, then there is such a transformation with $s\,=\,1$.
\end{lemma}

\section{Restriction to the moduli stack of semistable vector bundles}
\label{section:beyondGIT}

We will use similar ideas as in the proof of the Torelli Theorem for the moduli stack in \cite[\S~2]{ABGM}, 
namely, use \textit{the ``beyond GIT'' characterization of semistability} to prove that any isomorphism
between moduli stacks of vector bundles with fixed determinant preserves the locus of semistable vector bundles.

\begin{proposition}\label{prop:picm}
The Picard group of the moduli stack with fixed determinant is
generated by the determinantal line bundle
$$
\Pic(\SM(X,r,\xi))\ \cong\ \ZZ .
$$
If $g\,\geq\, 2$, the same holds for the semistable open substack 
$\SM^{\ssvb}(X,r,\xi)$, and if $g\,\geq \,3$, the same holds 
for the stable open substack $\SM^{\svb}(X,r,\xi)$.
\end{proposition}

\begin{proof}
In the case of moduli functor and moduli space, this was proved in \cite{DN}.
In the case of the moduli stack, a detailed proof can be found in
\cite[Proposition 4.2.3 and Theorem 4.2.1]{BH10}.

The statements for the semistable and stable substacks follow because
the moduli stack is smooth, and the restriction morphism of line bundles
to an open substack whose complement has codimension at least $2$ is 
an isomorphism. The codimension estimates
follow from Lemma \ref{lemma:codimNotSemiStable} and
Lemma \ref{lemma:codimDiv}.
\end{proof}

One of the differences between the moduli space and the moduli stack of
stable vector bundles is that the moduli stack always admits a
universal vector bundle but, if the degree and rank are not coprime, the
moduli space of stable vector bundles does not. Nevertheless, there is always
a covering by \'etale open subsets of the moduli space of stable bundles admitting a
Poincar\'e family.

\begin{lemma}\label{lem:covering}
There exists an \'etale covering $c\,:\,U\,\too\, M^{\svb}(X,r,\xi)$ such that
there is a tautological vector bundle $\SE_U$ on $X\times U$, i.e.,
the restriction $\SE\big\vert_{X\times \{u\}}$ is in the isomorphism class
corresponding to the point
$c(u)$ for all $u\, \in\, U$. The same holds for the moduli space of fixed degree 
$M^{\svb}(X,r,d)$. 
\end{lemma}

\begin{proof}
It is enough to give the proof in the fixed degree
case, because the fixed line bundle case follows by restriction.
The GIT construction of the moduli space shows that both the moduli
space and the moduli stack of stable bundles are given
as stack quotients (see for instance \cite{Si} and \cite[Proposition 3.3]{Gom}). There is a scheme $Q$ with an action of $\PGL(N)$
for $N$ sufficiently large such that
$$
M^{\svb}(X,r,d)\ \cong\ [Q/\PGL(N)], \quad \SM^{\svb}(X,r,d)\ \cong\ [Q/\GL(N)].
$$
The first isomorphism produces the following diagram:
$$
\xymatrix{
{U\times \GL(N)} \ar[r]^{h} \ar[rd]_{q} &{U\times \PGL(N)} \ar[r] \ar[d] &
{P} \ar[r]^{f} \ar[d]_{p} & Q\\
{} & {U} \ar[r]^-{c} & {M^{\svb}(X,r,d).}
}
$$
Indeed, giving a morphism from the moduli space to the stack quotient
is equivalent to giving a principal $\PGL(N)$-bundle $p$ on the moduli
space and a $\PGL(N)$-equivariant morphism $f$. The principal bundle
$p$ is locally trivial in the \'etale topology, so there is an \'etale
covering $c$ and the square in the diagram is Cartesian. The morphism
$h$ in the diagram is induced by the quotient map from $\GL(N)$ to $\PGL(N)$.

Note that the morphism $q$ in the diagram is a principal $\GL(N)$-bundle and the composition
of the morphisms in the first row is $\GL(N)$-equivariant. So we
obtain a morphism from $U$ to $[Q/\GL(N)]$ and $[Q/\GL(N)]$ is isomorphic to
the moduli stack, so it it induces a tautological vector bundle
$\SE_U$ on $X\times U$. 
\end{proof}

In the ``beyond GIT'' theory (see \cite{Al,He,HL,AlHLHe}), given a line bundle $\SL$ on a stack $\SM$, we say that a point
$x\,\in\, \SM$ is $\SL$-semistable if for all possible maps $f\,:\, [\Spec(\CC[t])/\CC^*] \,\longrightarrow\,
\SM$ such that $f(1) \,=\,x$, we have
$$\op{wt}(f^*\SL\big\vert_0)\ \le\ 0.$$

\begin{lemma}\label{lemma:recoverScheme}
Let $X$ and $X'$ be smooth complex projective curves of genus at least $2$. Let $\Psi\,:\,\SM(X,r,\xi)
\,\longrightarrow\, \SM(X',r',\xi')$ be an isomorphism of moduli stacks. Then $r\,=\,r'$, and $\Psi$ restricts to an isomorphism
$$\Psi^{\ssvb}\,:\,\SM^{\ssvb}(X,r,\xi)\,\longrightarrow\, \SM^{\ssvb}(X',r',\xi').$$
Furthermore, there exists an isomorphism of moduli schemes $\psi\,:\,M^{\ssvb}(X,r,\xi) \,\longrightarrow
\,M^{\ssvb}(X',r',\xi')$ such that the following diagram commutes:
\[
\xymatrix{
\SM^{\ssvb}(X,r,\xi) \ar[rr]^{\Psi^{\ssvb}} \ar[d]_{\pi} && \SM^{\ssvb}(X',r,\xi') \ar[d]^{\pi'}\\
M^{\ssvb}(X,r,\xi) \ar[rr]^{\psi} && M^{\ssvb}(X',r',\xi').
}
\]
\end{lemma}

\begin{proof}
{}From the Torelli Theorem for moduli stacks \cite{ABGM} it follows that $X\,\cong\, X'$ and $r\,=\,r'$. 
The Picard group of $\SM(X,r,\xi)$ is $\ZZ$ (Proposition \ref{prop:picm}) and it is generated by the determinantal line bundle $\SL_{\det}$. Then \cite[Proposition 3.3]{ABGM}
provides the following description of $\SL$-stability for any line bundle $\SL\,=\,\SL_{\det}^a\in \Pic(\SM(X,r,\xi))$:
\begin{itemize}
\item If $a\,<\,0$, then all points in $\SM(X,r,\xi)$ are $\SL_{\det}^a$-unstable.
\item If $a\,=\,0$, then all points in $\SM(X,r,\xi)$ are $\SL_{\det}^a$-semistable.
\item If $a\,>\,0$, then a point $x\,\in\, \SM(X,r,\xi)$ corresponding to a vector bundle $E$ is $\SL_{\det}^a$-semistable if and only if $E$ is semistable.
\end{itemize}
As a consequence, for any nontrivial line bundle $\SL$ on $\SM(X,r,\xi)$, the substack $\SM^{\ssvb}(X,r,\xi)$ can be canonically recovered from the moduli stack
$\SM(X,r,\xi)$ as follows. Choose any nontrivial line bundle $\SL$ on $\SM(X,r,\xi)$ and consider the substacks of $\SL$-semistable and $\SL^{-1}$-semistable
points. Exactly one of them is nonempty; this nonempty substack coincides with the substack of semistable vector bundles. Now, take any nontrivial
line bundle $\SL'$ on $\SM(X',r,\xi')$. So $\Psi^*\SL'$ is also nontrivial. The map $\Psi$ must send the locus of $\Psi^*\SL'$-semistable vector
bundles to the locus of $\SL'$-semistable vector bundles. In particular, one of them is nonempty if and only if the other is nonempty. Switching $\SL$ for $\SL^{-1}$ if
necessary, we can assume that both subschemes, namely the locus of $\SL'$-semistable vector
bundles and the locus of $\Psi^*\SL'$-semistable vector bundles,
are nonempty and, then, they must respectively be the corresponding substacks of semistable vector bundles $\SM^{\ssvb}(X,r,\xi)$ and $\SM^{\ssvb}(X',r',\xi')$. Therefore, the
map $\Psi$ must restrict to an isomorphism $\Psi^{\ssvb}\,:\,\SM^{\ssvb}(X,r,\xi)\,\longrightarrow\, \SM^{\ssvb}(X',r,\xi')$.

Finally, by \cite[Theorem 6.6]{Al}, the moduli scheme $M^{\ssvb}(X,r,\xi)$ is a good quotient of 
$\SM^{\ssvb}(X,r,\xi)$ which corepresents the moduli stack. Thus, there must exist a unique isomorphism 
$$\psi\ :\ M^{\ssvb}(X,r,\xi)\ \longrightarrow\ M^{\ssvb}(X',r,\xi')$$ such that the following diagram 
commutes
\[
\xymatrix{
\SM^{\ssvb}(X,r,\xi) \ar[rr]^{\Psi^{\ssvb}} \ar[d]_{\pi} && \SM^{\ssvb}(X',r,\xi') \ar[d]^{\pi'}\\
M^{\ssvb}(X,r,\xi) \ar[rr]^{\psi} && M^{\ssvb}(X',r',\xi') \, .
}\
\]
This completes the proof.
\end{proof}

Lemma \ref{lemma:autoScheme} (which corresponds to \cite[Theorem B]{KP95}, \cite[Theorem 1.1]{BGM13} and \cite[Theorem 2.11]{AB}) gives a full classification
of isomorphisms between moduli schemes of vector bundles with fixed determinant. Thus, we know all possible choices for the isomorphism $\psi$ obtained in
Lemma \ref{lemma:recoverScheme}. Given an isomorphism $\Psi$ as in Lemma \ref{lemma:recoverScheme}, if $X$ and $X'$ are curves of genus at least $4$ then there must exist
\begin{itemize}
\item an isomorphism $\sigma\,:\,X'\,\stackrel{\sim}{\longrightarrow} \,X$,
\item a sign $s\,\in\, \{1,\,-1\}$, and
\item a line bundle $L$ on $X$ such that $\xi'\,\cong\, \sigma^*(\xi\otimes L^r)^s$,
\end{itemize}
such that $\psi\,=\,\ST_{\sigma,L,s}$. We know that the basic transformation $\ST_{\sigma,L,s}$ gives a well defined isomorphism of stacks $\ST_{\sigma,L,s}
\,:\,\SM(X,r,\xi) \,\longrightarrow\, \SM(X',r',\xi')$. Composing with its inverse gives an automorphism
$$\Phi\ :=\ \ST_{\sigma,L,s}^{-1} \circ \Psi\ :\ \SM(X,r,\xi)
\ \longrightarrow\ \SM(X,r,\xi)$$ of the moduli stack of vector bundles of fixed determinant $\xi$ which induces the identity map on the corresponding
moduli scheme. Applying again Lemma \ref{lemma:recoverScheme} to this automorphism $\Phi$, we know that it must preserve the semistable locus and --- as it induces
the identity map on the moduli scheme --- it must also preserve the stable locus. The following lemma shows that the restriction of $\Phi$ to the moduli substack of stable bundles
is actually the identity map.

\begin{lemma}
\label{lemma:tensorLineBundle}
Let $X$ be a smooth complex projective curve of genus $g\ge 3$. Let 
$$
\Phi\,:\,\SM^{\svb}(X,r,\xi) \,\longrightarrow \,\SM^{\svb}(X,r,\xi)
$$
be an automorphism such that the induced map $\varphi\,:\,M^{\svb}(X,r,\xi) \,\longrightarrow\, M^{\svb}(X,r,\xi)$ coincides with the
identity map, i.e., $\Phi(E)\,\cong \,E$ for each stable vector bundle
$E$. Then $\Phi$ is isomorphic to the identity map.
\end{lemma}

\begin{proof}
In this proof we denote $\SM^{\svb}\,=\,\SM^{\svb}(X,r,\xi)$ and $M^{\svb}\,=\,M^{\svb}(X,r,\xi)$. Let $\SE\,=\,\SE_{\op{univ}}^{\svb}$ be the universal vector bundle on
$\SM^{\svb}$, and let $\SE'\,=\,(\id_X\times \Phi)^*\SE$. 

We claim that
$$
\SL\ =\ R^0(\pi_{\SM^{\svb}})_*\left(\SE^*\otimes \SE'\right)
$$
is a line bundle on $\SM^{\svb}$.

To prove the claim, let $s\,:\,Q\too \SM^{\svb}\,=\,[Q/\GL(N)]$ be an atlas (cf. proof of Lemma
\ref{lem:covering}). It suffices to prove that $s^*\SL$ is a line bundle. The morphism $s$
is flat, so by flat base change (cf. \cite[Lemma 005Y]{HL25}) we have
$$
s^* \SL\ =\ \pi_Q{}_* (\id_X\times s)^* \left(\SE^*\otimes \SE'\right).
$$
The scheme $Q$ is integral, so we can apply Grauert's theorem
(cf. \cite[Corollary III 12.9]{Ha}) to the morphism $\pi_Q$.
To do so, we calculate
$$
H^0((\id_X\times s)^* \left(\SE^*\otimes \SE'\right))\,=\,
\Hom(\SE\big\vert_{s(q)\times X}^\vee,\SE\big\vert_{\Phi(s(q))\times X})\,\cong\, \CC ,
$$
where the last isomorphism is due to the fact that $\SE\big\vert_{s(q)\times X}$ and $\SE\big\vert_{\Phi(s(q))\times X}$
are isomorphic stable bundles, as $\Phi$ induces the identity map at the level
of moduli schemes. Then Grauert's theorem says that $s^*\SL$ 
is a line bundle. This proves the claim.

Consider the natural evaluation morphism
\begin{equation}
\label{eq:evaluniv}
(\pi_{\SM^{\svb}})^*(\pi_{\SM^{\svb}})_*\left(\SE^*\otimes \SE'\right) \otimes
\SE\ \too\ \SE' .
\end{equation}
To check that it is an isomorphism, we pullback to $Q$ and use
base change. Then, for each slice $q\times X$ the evaluation morphism
becomes
$$
\Hom(\SE\big\vert_{s(q)\times X},\, \SE\big\vert_{\Phi(s(q))\times X})\otimes \SE\big\vert_{s(q)\times X}
\ \too\ \SE\big\vert_{\Phi(s(q))\times X}
$$
which is clearly an isomorphism because the Hom group is one-dimensional.
Therefore, the evaluation morphism in \eqref{eq:evaluniv} gives an isomorphism
\begin{equation}\label{eq:eprime}
(\pi_{\SM^{\svb}})^*\SL \otimes \SE \ \stackrel{\cong}{\too}\ \SE' .
\end{equation}
This implies that $\Phi$ is induced by the line bundle $\SL$, as
described in the introduction. But, as we are working with the moduli
stack of fixed determinant, we need $\SL^r\,\cong \,\SO_{\SM}$, 
and then $\SL\,\cong\, \SO_{\SM}$ because
$\Pic(\SM)\,=\,\Pic(\SM^{\svb}(X,r,\xi))\,\cong\, \ZZ$ is torsionfree (Proposition \ref{prop:picm}).
\end{proof}

\begin{lemma}
\label{lemma:extendLineBundle}
Let $g\geq 3$. Let $\SL$ be any line bundle over $\SM^{\svb}(X,r,\xi)$. Then there exists a unique extension of $\SL$ to $\SM(X,r,\xi)$.
\end{lemma}

\begin{proof}
The Picard groups of $\SM(X,r,\xi)$ and $\SM^{\svb}(X,r,\xi)$ are both $\ZZ$ (Proposition \ref{prop:picm}) and they are both generated by the determinantal line bundle,
which clearly extends to the whole moduli stack. Thus, any line bundle (which must be a multiple of the determinantal line bundle) extends uniquely.
\end{proof}

\section{Extension of morphisms to the complete moduli stack}
\label{section:extension}

Let $\SK$ be the field of rational functions of $X$. A divisor of rank $r$ and degree $d$ is a coherent sub $\SO_X$-module of $\SK^{\oplus r}$ of rank $r$ and degree $d$. The set of 
such divisors form an ind-scheme denoted by $\Div_X^{r,d}$ which is an atlas for the moduli stack of vector bundles $\SM(X,r,d)$. Each divisor $E\,\in\,
 \Div_X^{r,d}$ is a submodule of $\SO_X(D)^{\oplus r}$ for some effective divisor $D$. Let
$$\Div_X^{r,d}(D)\ =\ \{E\,\in\, \Div_X^{r.d} \,\, \big\vert \,\, E\,\subset\, \SO_X(D)^{\oplus r}\}.$$
Then
$$\Div_X^{r,d}\ =\ \bigcup_{D\ge 0} \Div_X^{r,d}(D).$$
Take any $D\,\ge\, 0$. Then $\Div_X^{r,d}(D)$ is isomorphic to the quot scheme 
$\Quot_{\SO_X(D)^r/X}^{r\deg(D)-d}$ and, clearly, if $D\,\le\, D'$ then there is a closed immersion 
$\Div_X^{r,d}(D)\,\hookrightarrow \,\Div_X^{r,d}(D')$, giving $\Div_X^{r,d}$ the above mentioned structure of 
ind-scheme. (See \cite{BGL94}.)

For each $D$, the scheme $\Div_X^{r.d}(D)$ is stratified by the Harder-Narasimhan type of the bundle. Recall that the Harder-Narasimhan
filtration of a vector bundle $E$ is a filtration of subbundles
$$0\,=\,E_0 \,\subsetneq\, E_1 \,\subsetneq \, \cdots \,\subsetneq\, E_k\,=\,E$$
such that $E_i/E_{i-1}$ is a stable vector bundle for all $1\, \leq\, i\, \leq\, k$ with
$$\frac{\rk(E_1/E_0)}{\deg(E_1/E_0)}\,>\,\frac{\rk(E_2/E_1)}{\deg(E_2/E_1)}\,>\,\ldots \,>\,\frac{\rk(E_k/E_{k-1})}{\deg(E_k/E_{k-1})}.$$
The sequence of points $(r_i,\, d_i)\,=\,\left( \rk(E_i/E_{i-1}),\, \deg(E_i/E_{i-1})\right)$ of ${\mathbb R}^2$ forms a strictly convex polygon called the Shatz
polygon for $E$. Let $\SP_{r,d}$ denote the set of possible Shatz polygons for a rank $r$ degree $d$ vector bundle; so $\SP_{r,d}$ is the set of strictly convex polygons
joining $(0,\,0)$ and $(r,\,d)$. For each $P\,\in \,\SP_{r,d}$, let $F_P(D)\,\subset \,\Div_X^{r,d}(D)$ denote the subset of divisors with Shatz polygon $P$. Then
$\Div_X^{r,d}(D)$ splits as follows:
$$\Div_X^{r,d}(D)\ =\ \coprod_{P\in \SP_{r,d}} F_P(D).$$
Since $\Div_X^{r,d}(D)$ represents a bounded family of vector bundles, the set of possible Shatz polynomials for the bundles in $\Div_X^{r,d}(D)$ is finite. Let
$\SP_{r,d}(X,D)$ denote this finite subset of $\SP_{r,d}$. Then
$$\Div_X^{r,d}(D) \ =\ \coprod_{P\in \SP_{r,d}(X,D)} F_P(D).$$

\begin{lemma}
\label{lemma:codimNotSemiStable}
Suppose that $g\,\ge\, 2$. Then there exists a constant $C(X,r,d)\,>\,0$
depending only on the curve, the rank $r$ and the degree $d$, such that
for each divisor $D$ with $\deg(D)\,\ge\, C(g,r,d)$ the codimension of the
locus of bundles in $\Div_{X}^{r,d}(D)$ which are not semistable is at least $2$.
\end{lemma}

\begin{proof}
Let $\SP_{r,d}'(X,D)\,:=\,\SP_{r,d}(X,D)\backslash \{ \left( (0,0),(r,d) \right)\}$ be the set of types which
are not semistable. Then the locus of non-semistable
vector bundles in $\Div_{X}^{r,d}(D)$ is
$$\Div_X^{r,d,ns}(D)\ =\ \coprod_{P\in \SP_{r,d}'(X,D)} F_P(D).$$
As this is a finite union of closed subschemes, it is enough to prove that $\op{codim}(F_P(D))\,\ge\, 2$ for each $P\,\in\, \SP_{r,d}'(X,D)$.

Let $P\,=\,\left ((0,\,0),\,\cdots,\,(r_i,\, d_i)\right)$ be a polygon corresponding to a non-semistable type in $\SP_{r,d}'(X,D)$.
By \cite[Proposition 5.2]{BGL94}, if $\deg(D)\,>\, \frac{d-d_i}{r-r_i}+2g-1$ then
$$\op{codim}(F_P(D))\ =\ \sum_{i>j} (r-r_i)(r-r_j) \left( \frac{d-d_j}{r-r_j} - \frac{d-d_i}{r-r_i} +g-1\right).$$
As $P$ forms a strictly convex polygon joining a $(0,\,0)$ with $(r,\,d)$, the sequence of slopes $\frac{d-d_i}{r-r_i}$ is strictly decreasing. Thus
$$\op{codim}(F_P(D))\, =\, \sum_{i>j} (r-r_i)(r-r_j) \left( \frac{d-d_j}{r-r_j} - \frac{d-d_i}{r-r_i} +g-1\right)>(g-1)\sum_{i>j}(r-r_i)(r-r_j)\, \ge\, g-1.$$
Since the codimension is an integer, we have $\op{codim}(F_P(D))\,\ge\, g \,\ge\, 2$.

On the other hand, again by \cite[Proposition 5.2]{BGL94}, if $\deg(D)\,\le\, \frac{d-d_i}{r-r_i}+2g-1$, then
$$\op{codim}(F_P(D))\ \ge\ \deg(D)-C$$
for some constant $C$ which only depends on $X$, $r$ and $d$, but neither on $P$ nor on the divisor $D$. Thus, if we assume that
$$\deg(D)\ \ge\ C+2,$$
then the codimension of $F_P(D)$ in $\Div_X^{r,d}(D)$ is at least $2$ in all the cases.
\end{proof}

\begin{lemma}\label{lemma:codimDiv}
Suppose that $g\,\geq\, 3$. Then there exists a constant $C(X,r,d)\,>\,0$
depending only on the curve, the rank $r$ and the degree $d$ such that
for each divisor $D$ with $\deg(D)\,\ge\, C(g,r,d)$ the codimension of the
locus of vector bundles in $\Div_{X}^{r,d}(D)$ which are not stable is at least $2$.
\end{lemma}

\begin{proof}
Let $Z^{\svb}\,\subset\, Z^{\ssvb}\,\subset\, \Div_{X}^{r,d}(D)$ be the (open) 
subschemes of points for which the underlying vector bundle is 
stable ($Z^{\svb}$) or semistable ($Z^{\ssvb}$); openness follows from \cite[p.~635, Theorem 2.8]{Ma}.

We need to estimate the dimension of $\Div_{X}^{r,d}(D)\setminus Z^{\svb}$. This set
is the disjoint union of $\Div_{X}^{r,d}(D)\setminus Z^{\ssvb}$ and 
$Z^{\ssvb}\setminus Z^{\svb}$.
By Lemma \ref{lemma:codimNotSemiStable},
the dimension of the first set satisfies the condition
$$
\dim \Div_{X}^{r,d}(D)\setminus Z^{\ssvb} \ \leq\ \dim \Div_{X}^{r,d}(D) -2.
$$
Therefore, to prove the lemma it is enough to show that
\begin{equation}\label{es1}
\dim Z^{\ssvb}\setminus Z^{\svb}\ \leq\ \dim \Div_{X}^{r,d}(D) -2.
\end{equation}

We will prove \eqref{es1} by constructing a scheme which parametrizes all divisors
of rank $r$ whose underlying vector bundle is strictly semistable.

Let $E$ be a strictly semistable bundle, and let
$$
0\,=\,E_0\,\subset\, E_1 \,\subset\, E_2\,\subset\, \cdots\, \subset\, E_l\,=\,E
$$
be the Jordan-H\"older filtration of $E$ \cite[p.~22, Definition 1.5.1]{HLe},
\cite[p.~22, Proposition 1.5.2]{HLe}. The successive quotients $Q_i\,=\,E_i/E_{i-1}$
are stable and have the same slope as $E$. The associated graded polystable vector bundle is
$$
{\rm gr}(E)\ :=\ \bigoplus_{i=1}^{l}Q_i.
$$
Combining the direct summands which are isomorphic, we can write
$$
{\rm gr}(E)\ =\ \bigoplus_{j=1}^{k}Q_j^{\oplus a_j},
$$
with $Q_j\,\not\cong\, Q_{j'}$ if $j\,\neq\, j'$.
The collection of triples $\{(\deg Q_j,\,\rk Q_j),\, a_j)\}_{1\leq j\leq k}$ 
is called the type of the semistable vector
bundle $E$. Note that there is only a finite number of types for all semistable bundles of fixed rank and degree.
To obtain a scheme parametrizing all semistable
bundles of a given type, we first take a space parametrizing the stable vector bundles 
$Q_j$. If the degree and rank of $Q_j$ are coprime, we take the moduli
space of stable vector bundles having that degree and rank. If they are not
coprime, then we take the \'etale covering given in Lemma \ref{lem:covering}.
Over the scheme parametrizing the quotients $Q_j$, we construct a scheme
parametrizing the successive extensions
$$
0\ \too\ E_{i-1}\ \too\ E_i\ \too\ Q_i\ \too\ 0.
$$
To estimate the dimension of the space of extensions, we use Riemann-Roch
and the fact that the dimension of the space of homomorphisms between two 
stable vector bundles of the same slope is one if they are isomorphic, and it
is zero if they are not.

The details of the computation are in the proof of \cite[Lemma 2.3]{BGM10},
and the result is that the dimension of $S$ satisfies the inequality
$$
\dim S\ \leq\ \dim M(X,r,d)-(r-1)(g-1).
$$
Note that the scheme $S$ constructed in this way parametrizes a tautological
family.

Using the fact that the set of semistable bundles is bounded, we can
find a constant $C(g,r,d)$ such that, if $\deg(D)\,>\,C(g,r,d)$ 
then $\Ext^1(E,\,\SO_X(D))\,=\,0$ for all semistable vector bundle $E$, and then
$\dim \Hom(E,\,\SO_X(D)^{\oplus r})$ is the same for all semistable vector
bundles $E$ and can be calculated by the Riemann-Roch
formula. We can then construct a scheme $\widetilde{S}$ which parametrizes
all divisors of rank $r$ with strictly semistable bundle $E$ and has
\begin{equation}\label{es2}
\dim \widetilde{S}\,\leq\, \dim M(X,r,\xi)+\dim\Hom(E,\SO_X(D)^{\oplus r})-1-(r-1)(g-1).
\end{equation}
Since $\widetilde{S}$ parametrizes all strictly semistable divisors of rank
$r$, we have
\begin{equation}\label{es3}
\dim Z^{\ssvb}\setminus Z^{\svb}\ \leq\ \dim \widetilde{S}.
\end{equation}
Now note that \eqref{es2} and \eqref{es3} together imply \eqref{es1}.
This completes the proof.
\end{proof}

As a direct consequence of Lemma \ref{lemma:codimDiv} it can be shown that every vector bundle can be 
approximated by a smooth 2-dimensional family of stable vector bundles; this is done below.

\begin{corollary}
\label{cor:2approximation}
For any vector bundle $E$ on a curve $X$ of genus $g\,\ge\, 3$, there exists a family of vector bundles $\SE$ on $X\times T$ for a certain smooth 2-dimensional scheme
$T$ and a point $t_0\,\in\, T$ such that $\SE\big\vert_{X\times \{t\}}$ is stable for each $t\,\neq\, t_0$ and $\SE\big\vert_{X\times \{t_0\}}\,\cong\, E$.
\end{corollary}

\begin{proof}
If $E$ is stable, then the statement is trivial. Otherwise, take a divisor $D$ of big enough degree so that $E$ embeds into $\SO_X(D)^{\oplus r}$ and such that $\deg(D)\,
\ge C(X,r,d)$, where 
$C(X,r,d)$ is the constant of Lemma \ref{lemma:codimDiv}. Then $E$ represents a non-stable point in $\Div_X^{r,d}(D)$. As $\Div_X^{r,d}(D)$ is smooth (see, for instance, \cite[p. 
6]{BGL94}), and the codimension of the non-stable locus is at least $2$ by Lemma \ref{lemma:codimDiv}, passing to an \'etale cover, if necessary, we can find a 2-dimensional smooth 
scheme $T$ mapping to a subset of $\Div_X^{r,d}$ crossing transversely the locus of non-stable bundles precisely at $E$. The pullback of the universal bundle on $\Div_X^{r,d}$ to $T$ 
gives the desired family.
\end{proof}

We shall now adapt the previous idea to work on arbitrary families of vector bundles with fixed determinant.

\begin{lemma}
\label{lemma:2approximation}
Let $\SE$ be a family of vector bundles on $X$ --- with $g\,=\, {\rm genus}(X) \, \geq\, 3$ --- of
rank $r$ and determinant $\xi$ of degree $d$ parameterized by a scheme $T$. Then there
exists a scheme $Q$ together with a vector bundle $\SF$ on $X\times Q$ and a morphism of schemes
$\tau\,:\,T\,\longrightarrow\, Q$ such that
\begin{itemize}
\item $Q$ is smooth,

\item the locus $Q^{ns}$ of points $q\,\in\, Q$ such that $\SF\big\vert_{X\times\{q\}}$ is non-stable has
codimension at least 2 in all the irreducible components of $Q$,

\item $\det(\SF)\,\cong\, \pi_X^* \xi$, where $\pi_X\,:\,X\times Q\,\longrightarrow\, X$ is the natural projection, and

\item $\SE\ \cong\ (\id_X\times \tau)^*\SF$.
\end{itemize}
\end{lemma}

\begin{proof}
Since $\SE$ is a bounded family, there exists a divisor $D$ such that $\SE\,\subset\, \pi_X^* \SO_X(D)^{\oplus r}$. Choose $D$ such that $\deg(D)
\,\ge\, C(X,r,d)$, where $C(X,r,d)$ is the bound given by Lemma \ref{lemma:codimDiv} (here we need $g\geq 3$). Fix once and for all one such embedding. Then, it
defines a map $f\,:\,T\,\longrightarrow \,\Div_X^{r,d}(D)$ such that $\SE\,\cong\, f^* \SE_{\op{univ}}$, where $\SE_{\op{univ}}$ is the universal vector
bundle on $\Div_X^{r,d}(D)$. Let us consider the map $\delta_\xi\,:\,\Div_X^{r,d}(D) \,\longrightarrow \,J(X)$ given by
$$\delta_\xi(E)\ :=\ \det(E) \otimes \xi^{-1}\, .$$
The preimage of $\SO_X$ is the set of divisors with determinant $\xi$. Let $\rho_r\,:\,J(X)\,\longrightarrow\, J(X)$ be the isogeny
$$\rho_r(L)\ =\ L^r,$$
and let $Q$ be the fibered product
\[
\xymatrix{
Q\,=\,\Div_X^{r,d}(D) \times_{J(X)} J(X) \ar[rr]^-{\pi_2} \ar[d]_{\pi_1} && J(X) \ar[d]^{\rho_r}\\
\Div_X^{r,d}(D) \ar[rr]^{\delta_\xi} && J(X).
}
\]
The variety $\Div_X^{r,d}(D)$ is isomorphic to the quot scheme
$\Quot_{\SO_X(D)^r/X}^{r\deg(D)-d}$, which, tensored by $\SO_X(-D)$,
is isomorphic to a quot scheme of $\SO_X^r$ and, therefore, it is
smooth and irreducible (see, for instance, \cite{EL}). 
The scheme $Q$ is smooth because base change of a smooth 
morphism is smooth \cite[Proposition III 10.1]{Ha}. 
By construction, $Q$ can be identified with the moduli space of pairs consisting of an inclusion $E\,\hookrightarrow\,
\SO_X(D)^{\oplus r}$ together with a line bundle $L$ on $X$ such that 
$$\det(E)\otimes \xi^{-1}\ \cong\ L^r .$$
Thus, given one such pair $(E,\,L)$, the tensor product $E\otimes L^{-1}$ has determinant $\xi$. Let $\SL_{\op{univ}}$ be the universal line
bundle on $X\times J(X)$. Then, by construction, the vector bundle
$$\SF\ =\ (\id_X\times \pi_1)^*\SE_{\op{univ}}\otimes (\id_X\times \pi_2)^*\SL_{\op{univ}}^{-1}$$
is a vector bundle on $X\times Q$ such that $\det(E)\,\cong\, \pi_X^*\xi$. Moreover, as $\det(\SE)\,\cong\, \pi_X^*\xi$,
the map $f\,:\,T\,\longrightarrow\, \Div_X^{r,d}(D)$ defined by the inclusion map
$\SE\,\hookrightarrow\, \SO_X(D)^{\oplus r}$ lifts to the map $\tau\,:\,T\,\longrightarrow\, Q$ defined by the pair
$\left(\SE\subset \SO_X(D)^{\oplus r},\, \pi_X^*\SO_X\right)$. Then, $(\id_X\times (\tau\circ \pi_2))^*\SL_{\op{univ}}\,=\,\SO_{X\times T}$. Thus,
$$(\id_X\times \tau)^*\SF\ \cong\ (\id_X\times f)^*\SE_{\op{univ}}\ \cong\ \SE .$$
Therefore, it is enough to prove that the locus $Q^{ns}$ of points $q\,\in\, Q$ such that $\SF\big\vert_{X\times \{q\}}$ is non-stable has codimension
at least 2 in $Q$. Since tensoring by a line bundle does not change stability, this is the same as the
locus of points such
that $(\id_X\times \pi_1^*)\SE_{\op{univ}}\big\vert_{X\times \{q\}}$ is not stable, which is the preimage $\pi_1^{-1}(\Div_X^{r,d,ns}(D))$ of the set of
non-stable points in $\Div_X^{r,d}(D)$. We have seen that $\Div_X^{r,d}(D)$ is irreducible and the fibers of the isogeny $\rho_r$ are
isomorphic to the set of $r$ torsion points of the Jacobian, which is finite. Thus, $\pi_1$ is a finite map, and the codimension of $Q^{ns}$
in all the irreducible components of $Q$ is the same as the codimension of $\Div_X^{r,d,ns}(D)$ in $\Div_X^{r,d}(D)$, which is at least $2$ by Lemma \ref{lemma:codimDiv}. 
\end{proof}

\begin{lemma}
\label{lemma:extensionId}
Suppose that $X$ is a curve of genus at least $3$. Let $\Phi\,:\,\SM(X,r,\xi)
\, \longrightarrow \, \SM(X,r,\xi)$ be an automorphism of the moduli stack such that the restriction of $\Phi$ to $\SM^{\svb}(X,r,\xi)$ is isomorphic to 
the identity
map. Then $\Phi$ is isomorphic to the identity.
\end{lemma}

\begin{proof}
For notational simplicity, let us write $\SM\,=\,\SM(X,r,\xi)$ and $\SM^{\svb}\,=\,\SM^{\svb}(X,r,\xi)$. Let $\SE$ be the universal vector bundle on $X\times \SM$. Let $\SE'
\,=\,\Phi^*\SE$. By Lemma \ref{lem:equivautovb}, to prove that 
$\Phi$ is isomorphic to the identity we have to prove that $\SE\,\cong\, \SE'$.

By hypothesis, $\SE\big\vert_{\SM^{\svb}}\,\cong\, \SE'\big\vert_{\SM^{\svb}}$. Let $f\,:\,T\,\longrightarrow\, \SM$ be any map. Let us prove that $f^*\SE\,\cong\, f^*\SE'$. Apply Lemma 
\ref{lemma:2approximation} to the family of vector bundles $f^*\SE$ on $T$. Let $\SF$ be the family on the normal scheme $Q$ given by Lemma \ref{lemma:2approximation}. Let $\tau\,:\,T
\,\longrightarrow\, Q$ be the 
map such that $f^*\SE \,\cong\, \tau^*\SF$. Then, $\SF$ defines a map $\widetilde{f}\,: \,Q \,\longrightarrow\, \SM$. Composing with the automorphism $\Phi$, this map is sent
to a map $$\widetilde{f}'\,:=\, 
\Phi\circ \widetilde{f} \, :\, Q\,\longrightarrow \,\SM .$$ This map is defined by the family of vector bundles
$$\SF'\ :=\ \widetilde{f}^* \SE' .$$
We will now prove that $\SF\,\cong\, \SF'$. Let $Q^{\svb}\,=\,Q\backslash Q^{ns}$ be the locus of points $q\,\in\, Q$ such that $\SF\big\vert_{X\times \{q\}}$ is stable. As
the map $\Phi$ is assumed to restrict to the identity map on the substack $\SM^{\svb}$, the family $\SF\big\vert_{X\times Q^{\svb}}$ and its image $\SF'\big\vert_{X\times Q^{\svb}}$ must
be isomorphic. By Lemma \ref{lemma:2approximation}, $Q^{\svb}$ is dense in $Q$ and the complement $Q^{ns}$ has codimension at least 2 in each component of $Q$.
The scheme $Q$ is normal and, therefore, it is Serre S2. Thus, the vector bundle $\SF\big\vert_{X\times Q^{\svb}}$ admits at most one extension to $X\times Q$. As $\SF$ and $\SF'$
are two such extensions, it follows that $\SF\,\cong\, \SF'$.

Now, since $\Phi$ is an isomorphism of stacks, the family
$$f^*\SE\ =\ (\widetilde{f}\circ \tau)^*\SE\ \cong\ \tau^* \SF$$
is sent to the family
$$f^*\SE'\ =\ (\widetilde{f}\circ \tau)^*\SE'\ \cong\ \tau^* \SF'\ \cong\ \tau^*\SF\ \cong\ f^*\SE .$$
Thus $f^*\SE\,\cong\, f^*\SE'$. Since this holds for each map $f\,:\,T\,\longrightarrow \,\SM$, it follows that
$\Phi$ must be isomorphic to the identity.
\end{proof}

\begin{theorem}\label{thm:main}
Let $X$ and $X'$ be smooth complex projective curves of genus $g$ and $g'$ respectively, with $g,\, g'\, \ge\, 4$. Let $\xi$ and $\xi'$ be line bundles on $X$ and $X'$
respectively. Let $r,\,r'\,>\,0$. Let $\Psi\,:\, \SM(X,r,\xi)\,\longrightarrow\, \SM(X',r',\xi')$ be an
isomorphism between the moduli stacks. Then $r\,=\,r'$, and there exists
\begin{itemize}
\item an isomorphism $\sigma\,:X'\,\too\, X$,
\item a sign $s\,\in \, \{+1,\, -1\}$, which can always be taken as $+1$ if $r\,=\,2$,
\item a line bundle $L$ over $X$ together with an isomorphism $\xi'\,\cong\, \sigma^*(\xi\otimes L^r)^s$
\end{itemize}
such that the map $\Psi$ is isomorphic to the basic transformation $\ST_{\sigma,L,s}$ defined in Definition \ref{def:basicTransformation}.
\end{theorem}

\begin{proof}
By Lemma \ref{lemma:recoverScheme}, the map $\Psi$ preserves the substack of semistable vector bundles and there exists an isomorphism
$\psi\,:\,M^{\ssvb}(X,r,\xi) \,\longrightarrow\, M^{\ssvb}(X,r,\xi)$ such that the following diagram commutes:
\[
\xymatrix{
\SM^{\ssvb}(X,r,\xi) \ar[rr]^{\Psi^{\ssvb}} \ar[d]_{\pi} && \SM^{\ssvb}(X',r,\xi') \ar[d]^{\pi'}\\
M^{\ssvb}(X,r,\xi) \ar[rr]^{\psi} && M^{\ssvb}(X',r',\xi') \, .
}\
\]
By Lemma \ref{lemma:autoScheme}, there exists an isomorphism
$\sigma\,:\,X\,\too\, X'$, a line bundle $L$ over $X$ and $s\,\in\, \{\pm 1\}$ such
that $\xi'\,\cong\, \sigma^*(\xi\otimes L^r)^s$ and
$\psi\,=\,\ST_{\sigma,L,s}$. Moreover, $s\,=\,1$ if $r\,=\,2$. Now $\ST_{\sigma,L,s}$
extends to an automorphism of the stack. Composing with the inverse of
$\ST_{\sigma,L,s}$ we obtain an automorphism
$$\Phi\,:=\,\ST_{\sigma,L,s}^{-1} \circ \Psi\,:\, \SM(X,r,\xi) \,\longrightarrow\,
\SM(X,r,\xi)$$ which, again by Lemma \ref{lemma:autoScheme}, preserves
the substack $\SM^{\ssvb}(X,r,\xi)$, but which now descends to the
identity map on the moduli scheme $M^{\ssvb}(X,r,\xi)$. In particular, it
also preserves the substack of stable vector bundles
$\SM^{\svb}(X,r,\xi)$. By Lemma \ref{lemma:tensorLineBundle}, we know
that the restriction of it to $\SM^{\svb}(X,r,\xi)$ is isomorphic to the identity map,
and that the identity extends uniquely to the identity on $\SM(X,r,\xi)$ by
Lemma \ref{lemma:extensionId}.
Therefore, we have
$$\Psi\ \cong \ \ST_{\sigma,L,s} .$$
This completes the proof.
\end{proof}

Theorem \ref{thm:main} describes the objects of the 
category of isomorphisms between moduli stacks or, equivalently, the 
automorphisms of a given moduli stack. It says that the set of 
isomorphism classes of such automorphisms coincides with the elements
of the group $\op{Aut}(M^{\ssvb}(X,r,\xi))$. The category is described
in the following theorem.

\begin{theorem}
\label{thm:main2}
If $X$ is a smooth complex projective curve of genus at least $4$, then
there is a canonical equivalence of groupoids
$$\op{Aut}(\SM(X,r,\xi))\ \cong\ \big[ \, 
 \op{Aut}(M^{\ssvb}(X,r,\xi))\, / \, (\ZZ/r\ZZ) \, \big]$$
where the action of $\ZZ/r\ZZ$ on $\op{Aut}(M^{\ssvb}(X,r,\xi))$
is the trivial action. Note that $\op{Aut}(M^{\ssvb}(X,r,\xi))$
is a discrete finite group, so the stack $\op{Aut}(\SM(X,r,\xi))$
is just the disjoint union of a finite number of copies
of $B \ZZ/r\ZZ$.

\end{theorem}

\begin{proof}
The objects of the groupoid $\op{Aut}(\SM(X,r,\xi))$ of automorphisms
of the moduli stack are described in Theorem \ref{thm:main}.
Using Lemma \ref{lem:equivautovb}, the automorphism group of each object
in this category is identified with the automorphism group of the corresponding
vector bundle on $X\times \SM$, and this group is isomorphic to $\ZZ/r\ZZ$ 
(Proposition \ref{prop:autmap}).
\end{proof}

\section{The case of projective line}
\label{section:rational}

In higher genus, Theorem \ref{thm:main} says that there is a bijection
between the automorphisms of the moduli scheme and the isomorphism classes of
automorphisms of the moduli stack. In the case $X\,=\,\PP^1$ this is not the case. 
For simplicity, we assume that rank $r\,=\,2$ and $\xi\,=\,\SO_{\PP^1}$. A vector bundle
on $\PP^1$ with these invariants is of the form $\SO_{\PP^1}(-n)\oplus \SO_{\PP^1}(n)$ for some $n\in \NN$. It is semistable when $n\,=\,0$, so the moduli scheme
of semistable vector bundles is just one point, so its automorphism group
is trivial.

Consider the action of the automorphism group of $\PP^1$ 
(which is isomorphic to $\PSL(2,\CC)$) on the moduli stack. If
$\sigma$ is an automorphism of $\PP^1$, define $F_\sigma$
as the functor which sends a family of vector bundles $E\,\too\, \PP^1\times T$ to 
$(\sigma\times \id)^*E$ (the action on morphisms is clear).

This action is trivial at the level of isomorphism classes, because 
the pullback of any line bundle on $\PP^1$ 
is isomorphic to itself. But the action is not trivial on morphisms.
An automorphism of $\SO_{\PP^1}(-n)\oplus \SO_{\PP^1}(n)$ defined by the matrix
$$
\left (
\begin{array}{cc}
\lambda & 0\\
s & \mu \\
\end{array}
\right )
$$
where $\lambda,\, \mu\,\in\, \CC^*$ and $s$ is a section of $\SO_{\PP^1}(2n)$. 
The functor $F_{\sigma}$ sends
this morphism to the morphism defined by the matrix 
$$
\left (
\begin{array}{cc}
\lambda & 0\\
\sigma^* s & \mu \\
\end{array}
\right )
$$
If $\sigma$ does not preserve the zeros of $s$, then
it
is easy to check that $F_{\sigma}$ is not isomorphic to the identity.
Therefore, the action of $\PSL(2,\CC)$ on $\SM(\PP^1,2,\SO_{\PP^1})$ is
not trivial.

\section*{Acknowledgments} 

We thank Andr\'es Fern\'andez Herrero for discussions and several suggestions
that helped simplify some of the proofs. This research was supported by grants
PID2022-142024NB-I00, RED2022-134463-T and CEX-2023-001347-S funded by MCIN/AEI/ 
10.13039/501100011033. The second-named author is partially supported by a 
J. C. Bose Fellowship (JBR/2023/000003). The third-named author thanks Shiv Nadar
University for hospitality during a visit where part of this work was done.

\end{document}